\documentclass[final]{siamltex}

\usepackage{epsfig}
\usepackage{graphicx}
\usepackage{subcaption}
\usepackage{amsmath,amsfonts,amssymb}
\usepackage{algpseudocode,algorithm}
\newcommand{\smat}[1]{ \left[\begin{smallmatrix} #1 \end{smallmatrix}\right]}

\def\setC{\,\mathbb{C}}
\def\setR{\mathbb{R}}

\newtheorem{rumatemppu}{Example}
\newenvironment{example}{\begin{rumatemppu}\rm}{\end{rumatemppu}}

\title{Fastest quotient iteration 
 with variational principles 
for self-adjoint eigenvalue problems}

\author{
Marko Huhtanen\thanks{
Faculty of Information Technology and Electrical Engineering,
University of Oulu,
90570 Oulu 57,
Finland,
({\tt Marko.Huhtanen@aalto.fi}).}
\and
Vesa Kotila \thanks{
Faculty of Information Technology and Electrical Engineering,
University of Oulu,
90570 Oulu 57,
Finland,
({\tt Vesa.Kotila@oulu.fi}).
}
\and
Pauliina Uusitalo \thanks{
Faculty of Information Technology and Electrical Engineering,
University of Oulu,
90570 Oulu 57,
Finland,
({\tt Pauliina.Uusitalo@oulu.fi}).}
}
\begin{document}
\maketitle
\begin{abstract} For the generalized eigenvalue problem, a quotient function is devised for estimating 
eigenvalues in terms of an approximate eigenvector. 
This gives rise to an infinite family of 
quotients, all entirely arguable to be used in estimation. Although the Rayleigh quotient is among them, one can suggest using it only in an auxiliary manner for choosing the quotient
for near optimal results.  In normal eigenvalue problems, for any approximate eigenvector, 
there always exists a "perfect" quotient exactly giving an eigenvalue. 
For practical estimates in the self-adjoint case, an approximate midpoint of the spectrum is a good choice for 
reformulating the eigenvalue problem yielding apparently 
the fastest quotient iterative method there exists.
No distinction is made between estimating extreme or interior eigenvalues. 
Preconditioning from the left results in changing the inner-product 
and affects the estimates accordingly. 
Preconditioning from the right preserves self-adjointness and can hence be performed without any restrictions. 
It is used in  variational methods for optimally computing approximate eigenvectors. 
\end{abstract}



\pagestyle{myheadings}
\thispagestyle{plain}

\markboth{M. HUHTANEN, V. KOTILA AND P. UUSITALO
}{FASTEST QUOTIENT ITERATION}

\section{Introduction}
Based on inspecting all conceivable quotients, this paper is concerned with the computation of best eigenvalue estimates for a large nonsingular\footnote{Nonsingular means there are nonsingular linear
combinations of the matrices $M$ and $N$.} self-adjoint eigenvalue problem
\begin{equation}\label{gen}
Mx=\lambda Nx,
\end{equation}
with matrices $M,N\in \setC^{n \times n}$. Large means, typically, that only a few specific eigenpairs are meant to be computed. Self-adjoint means that there
is an inner-product 
$(x,y)_P=(Px,y)=y^*Px$ available,  defined in terms of a
positive definite matrix $P\in \setC^{n \times n}$, such that
\begin{equation}\label{sadjo}
(Mx,Nx)_P\;
 \mbox{is real for any}\, x \in \setC^n,
 \end{equation}
i.e., $N^*PM$ is a Hermitian matrix  
\cite{HOOKOO,HNEAR}.\footnote{This should not be confused with
the classical notion of Hermitian matrix pencil $M-\lambda N$ involving two
Hermitian matrices $M$ and $N$. 
For Hermitian matrix pencils, see
\cite{PA,PASYM}.} These assumptions are met in many realistic applications.
The standard Hermitian eigenvalue problem corresponds to having $P=I$ and
$N=I$. Non-Hermitian quantum mechanics \cite{BB} is  concerned with the case $P\not=I$.
In numerical computations $P\not= I$ usually appears in two ways.
First, the inner-product may depend on the chosen discretization when using,
e.g., the finite element method (FEM).
Second, in preconditioning the eigenvalue problem from the left, 
the inner-product gets changed in computing estimates; see \eqref{prer}. 
(For computational aspects of the eigenvalue problem, see, e.g., \cite{PA,SAAD,IKRA} 
and the large number of references therein.)
Thus, assume $x\in \setC^n$ an approximate eigenvector is available
and the task is to estimate eigenvalues.
Then, in the self-adjoint case, the optimal quotients 
\begin{equation}\label{qtti}
{\rm oq}_{M,N}(x)=\frac{(Mx,Nx)_P}{|(Mx,Nx)_P|}
\frac{\|Mx\|_P}{\|Nx\|_P}   
\end{equation}
and the Rayleigh quotients
\begin{equation}\label{ray}
{\rm rq}_{M,N}(x)=\frac{(Mx,Nx)_P}{(Nx,Nx)_P}
\end{equation}
are real such that respective quotient iterations attain cubic speed of convergence \cite{HUKO}.
Let $\mu \in \setC$. For Rayleigh quotients we have
${\rm rq}_{M-\mu N,N}(x)={\rm rq}_{M,N}(x)- \mu$ while  
optimal quotients behave non-linearly in such translations.
This non-linearity turns out not to be a disruption.
That is, to simultaneously cover all reasonable eigenvalue estimates,
define the quotient function
\begin{equation}\label{funa}
\mu \longmapsto {\rm oq}_{M-\mu N,N}(x)+\mu 
\end{equation}
on $\setC$. 
First and foremost, the entire spectrum can be recovered 
with the quotient function 
if both $x$ and $\mu$ are allowed vary \cite[Theorem 3.2]{HOOKOO}.
Second, it yields the field of values at 
infinity \cite[Theorem 3.3]{HOOKOO} making, intriguingly, the Rayleigh quotients redundant as
\begin{equation}\label{infy}
\lim_{\mu \rightarrow \infty}
{\rm oq}_{M-\mu N,N}(x)+\mu=
{\rm rq}_{M,N}(x).
\end{equation}
In the self-adjoint case, by 
using the growth properties of the quotient function  \eqref{funa} when $\mu\in\setR$,  
superior eigenvalue estimates over the classical ones are derived. 
For classical estimates, see the concise description \cite{PLU} and references 
therein and \cite[p. 75]{PA}.\footnote{This is in contrary to what is taught in textbooks on numerical methods. Typically the Rayleigh quotients are claimed to be the "best choice".}
These quotients are then used to have the fastest quotient iteration
we are aware of, beating also the Rayleigh quotient iteration. To start the iteration with 
a high quality starting vector to immediately attain a cubic speed of convergence, 
preconditioned variational methods for optimally generating approximate eigenvectors are devised.  
Guaranteed estimates are obtained in terms of a shrinking sequence of intervals, each containing an eigenvalue, converging to an eigenvalue.
 
The closure of the image of the quotient function \eqref{funa}
is denoted by ${\rm Imqf}_{M,N}(x)$.
It is shown to be the disc of radius $\frac{\|Mx-{\rm rq}_{M,N}(x) Nx\|_P}{\|Nx\|_P}$
centred  at ${\rm rq}_{M,N}(x)$.
Consequently,  to estimate eigenvalues,
the Rayleigh quotient can be interpreted as being the average of
all reasonable quotients associated with an approximate eigenvector 
$x\in \setC^n$.
If the eigenvalue problem is normal (see Definition \ref{norsu}), then ${\rm Imqf}_{M,N}(x)$ contains an eigenvalue as
there exists a "perfect" quotient \eqref{funa} giving
an eigenvalue exactly with an appropriate choice of $\mu$; see Corollary
\ref{inklu}.  Conversely, and very discouragingly,  if the eigenvalue problem is not normal, there may not exist any  reasonable quotient to be used in estimation.
To estimate an extreme eigenvalue in the self-adjoint case,
rather than taking the Rayleigh quotient, a simple rule 
can be given based on the growth properties of the quotient function as follows.
For the smallest eigenvalue $\lambda_1$ (resp. the largest $\lambda_n$) the parameter $\mu$ should
satisfy $\mu> {\rm rq}_{M,N}(x)$ 
(resp. $\mu< {\rm rq}_{M,N}(x)$). 
The issue arises, how close to ${\rm rq}_{M,N}(x)$ 
should  $\mu$ be brought
to beat the Rayleigh
quotient by a good margin.
Improved estimates can be guaranteed with $\mu$ being the midpoint of the spectrum, i.e.,  
$|\lambda_1-\mu|=|\lambda_n-\mu|$. Then
\begin{equation}\label{erwr}
(M-\mu N)x=\lambda Nx
\end{equation}
is a natural formulation to produce estimates.
Of course, this condition for $\mu$ can be satisfied only roughly, based on incomplete information about the interval containing the eigenvalues.  
This is sufficient for practical purposes, though. Most notably, estimates in positive semi-definite problems  can be improved without any additional information; see Corollary \ref{pual} and Algorithm \ref{alg:pyor}.
Once an approximate midpoint $\mu$ has been set, every subsequent quotient iteration will be speeded-up due to these improved quotients. This yields a quotient iteration
which we argue to be fastest there exists; see Algorithm  \ref{alg:factspar}. For the background and 
current use of the Rayleigh quotient iteration, see  \cite{PA} and  \cite[p.  194]{SAAD}. See also \cite{WUSS} (in standard eigenvalue problems) and its Jacobi-Davidson variants derived by using Newton's method, see \cite[pp. 206--209]{SAAD} and references therein. (Bear in ming that the Jacobi-Davidson methods cannot compete with the Rayleigh quotient iteration unless exact inversion are performed; see \cite[p. 604]{SBFV}, \cite{FS} and \cite[p. 218]{SAAD}.) 
For interior eigenvalues, near a given point $\zeta \in \setR$, the approach is the same once the eigenvalue problem
is reformulated as
\begin{equation}\label{ilahi}
Nx=\lambda (M-\zeta N)x,
\end{equation} 
converting the associated interior eigenvalues into extreme.
In particular, computationally there is no distinction between interior and exterior eigenvalues.  For interior eigenvalue problems, e.g.,  in quantum physics, see \cite{ES,ELS} and references therein.

All the preceding estimates with quotients require high quality
approximate eigenvectors.
Generation of approximate eigenvectors is actually the most challenging part of any process to estimate eigenvalues. For starting the
Rayleigh quotient iteration to attain convergence towards a desired eigenvalue,
this dilemma is well-known; 
see \cite[pp. 84--85]{PA}.
For a numerical experiment illustrating this unpredictablity of convergence,
see \cite[Example 3.1.]{HUKO}. Also, rapidly
attaining a cubic speed of convergence requires
a good starting vector.
In the self-adjoint case we take the quotient function minus $\mu$ and vary $x$ by imposing the optimality condition
\begin{equation}\label{aminni}
\min_{x \in \setC^n}\frac{\|(M- \mu N)x\|_P^2}{\|Nx\|_P^2}
\end{equation}
whose solution gives an eigenvector corresponding to the eigenvalue nearest to $\mu$. 
To make this practical, restrict $x\in \setC^n$ to belong to
small dimensional subspaces generated by a preconditioned descent method.
That is, since self-adjointness is preserved in applications by an invertible 
$Z\in \setC^{n \times n}$ from the right, the speed of descent can be affected$/$increased by considering 
\begin{equation}\label{aminnai}
\min_{x \in \setC^n}\frac{\|(M- \mu N)Zx\|_P^2}{\|NZx\|_P^2}
\end{equation}
after preconditioning. We suggest monitoring the evolution of estimates in terms of ${\rm Imqf}_{M,N}(x)$, for guaranteed inclusion regions. A shrinking sequence of intervals,
each containing an eigenvalue, converging to an eigenvalue is obtained.
From each interval it is possible to pick a quotient for an eigenvalue estimate. Then swapping the optimal quotient iteration Algorithm  \ref{alg:factspar}, we expect convergence in two or at most three iterations.

The paper is organized as follows. In Section \ref{Sec2} the quotient function is introduced. Its image is shown to contain an eigenvalue in normal eigenvalue problems while in general this may not hold without changing the inner-product. In Section \ref{Sec3} the growth properties of the quotient function are inspected in the self-adjoint case. Rules are given to have high quality quotients, resulting in the fastest quotient iteration we are aware of.  
The issue of providing good starting vectors for quotient iterations is adressed in Section \ref{Sec4}. A preconditioning 
strategy is described. In Section \ref{Sec5} the problem of estimating several eigenvalues is addressed. Numerical experiments, concerned with realistic and tough problems in computational physics, are conducted in Section \ref{Sec6}.

\section{Quotient function for eigenvalue estimation}\label{Sec2}
Practically any algorithm for computing eigenvalues, all or just a few, relies on estimating extreme eigenvalues  using approximate eigenvectors. If the eigenvalue problem is standard, the power method combined with computing  Rayleigh quotients \eqref{ray}
for the largest eigenvalue is a fundamental example of this.
These ideas are then used in a repetitive manner for other eigenvalues. Rayleigh quotients were initially used, primarily by physicists, to estimate smallest eigenvalues of 
a self-adjoint positive definite operator $M$. 
Typically this took place with the Laplace operator  \cite{RAY, KATO,MA,BPAR}.  
In the positive definite case the Rayleigh quotients are actually optimal quotients \eqref{qtti} without having to take the limit \eqref{infy}.
Namely, with an approximate eigenvector $x$, take the square root $M^{1/2}$ of $M$ to have
\begin{equation}\label{sqr}
	{\rm rq}_{M,I}(x)
	=  {\rm oq}_{M^{1/2},I}(x)^2.
\end{equation}
That is, for a positive definite operator $M$,  the Rayleigh quotients 
are optimal quotients of $M^{1/2}$ squared. For further estimates in this case, see Example \ref{ubo} below.

Besides Rayleigh quotients, there are many other options to produce
estimates in terms of  
an approximate eigenvector. The quotient function defined as follows provides a way to simultaneously collect all conceivable quotients for estimation.
For the rationale behind its derivation in the case $\mu=0$, see \cite[Section 2.1]{HUKO} and \cite{HNEAR}.
See also Appendix A.

\begin{definition}
Let $M,N\in \setC^{n \times n}$. For a positive definite 
$P\in \setC^{n\times n}$ determining the inner-product, 
the quotient function of the eigenvalue problem 
\eqref{infy} at $x \in\setC^n$  with $Nx\not=0$ is  
$$\mu \longmapsto 
\frac{((M-\mu N)x,Nx)_P}{|((M-\mu N)x,Nx)_P|}
\frac{\|(M-\mu N)x\|_P}{\|Nx\|_P} +\mu,$$
defined for  $\mu \in \setC \smallsetminus \{{\rm rq}_{M,N}(x)\}$.
\end{definition}

Like with the Rayleigh quotient, two matrix-vector products need to be performed to evaluate the quotient function, plus the cost of applying $P$. In practice this cost depends on the discretization determining, e.g., the sparsity of $P$.

\smallskip

\begin{example}\label{pde}
For a common example, a discretization of a PDE using the finite element method (FEM) often leads to an eigenvalue problem \eqref{gen}
with $M$ and $N$ Hermitian such that, given blockwise, we have
\begin{equation}\label{sar}
        \left[ \begin{array}{cc}
                M_{11}&M_{12}\\
                M_{12}^*&M_{22}
        \end{array} \right]\left[ \begin{array}{c}
        x_{1}\\
        x_{2}
        \end{array} \right]
        =\lambda
\left[ \begin{array}{cc}
        0&0\\
        0&N_{22}
\end{array} \right]\left[ \begin{array}{c}
x_{1}\\
x_{2}
\end{array} \right];
\end{equation}
see \cite{BEBO}. If $M_{11}$ is invertible
and $N_{22}$ positive definite,
an inner-product can be constructed by setting
$Z=\smat{I&0\\-M_{12}^*M_{11}^{-1}&I}$.
Then $$P=Z^*\left[ \begin{array}{cc}
I&0\\ 0&N_{22}^{-1}
\end{array} \right]
Z$$ yields an inner-product such that \eqref{sadjo} holds.  
\end{example}

\begin{example} In the preceding PDE discretization example, 
an often encountered case is $N=N_{22}$  being the mass matrix. (This takes place with  the Laplacian eigenvalue problem, for instance.)
The cost of applying $P=N^{-1}$ requires inversions
either by invoking sparse direct solvers or using iterative methods.
Then, in this inner-product, the Rayleigh quotient \eqref{ray} equals $$\frac{(Mx,x)}{(Nx,x)}$$
 which is typically called  {\em the} Rayleigh quotient; see, e.g., \cite{PA}. It is, however, just the 
appearance of the Rayleigh quotient \eqref{ray}  using this particular inner-product.   (If there is one, then there is an infinite number of
eligible inner-products available \cite{HNEAR}.)
\end{example} 
 
\smallskip

The image of the quotient function admits the following characterization.  

\begin{theorem}\label{pallero} 
The closure of the image of the quotient function is the 
disc of radius
$\frac{\|Mx-{\rm rq}_{M,N}(x) Nx\|_P}{\|Nx\|_P}$ centred 
at ${\rm rq}_{M,N}(x)$.
The limit of the quotient function at infinity is
 ${\rm rq}_{M,N}(x)$ while 
the limit at ${\rm rq}_{M,N}(x)$ is the 
circle of radius 
$\frac{\|Mx-{\rm rq}_{M,N}(x) Nx\|_P}{\|Nx\|_P}$ centred 
at ${\rm rq}_{M,N}(x)$.
\end{theorem}

\begin{proof} If $x$ is an eigenvector, then the claim is true with the quotient function being constant, i.e., the respective eigenvalue. 

So let us assume $x$ is not an eigenvector.
Since $Nx\not=0$, the quotient function
can be discontinuous only when the first fraction is discontinuous. This takes place at its zeroes. There is just one, so let us concentrate on 
${\rm rq}_{M,N}(x)$.
To analyze this discontinuity,
denote $\mu={\rm rq}_{M,N}(x)+ \hat{\mu}$. Then
$$\frac{((M-\mu N)x,Nx)_P}{|((M-\mu N)x,Nx)_P|}=
-\frac{\hat{\mu}}{|\hat{\mu}|}$$
and 
$$\frac{\|(M-\mu N)x\|_P}{\|Nx\|_P}=
\frac{\sqrt{\|Mx-{\rm rq}_{M,N}(x)Nx\|_P^2+
|\hat{\mu}|^2\|Nx\|_P^2}}{\|Nx\|_P}.$$
So the quotient function takes the form
\begin{equation}\label{sade}
\mu \longmapsto
-\frac{\hat{\mu}}{|\hat{\mu}|}
\sqrt{\frac{\|Mx-{\rm rq}_{M,N}(x)Nx\|_P^2}{\|Nx\|_P^2}+
|\hat{\mu}|^2} +{\rm rq}_{M,N}(x)+ \hat{\mu}
\end{equation}
With $\hat{\mu}$ given in the polar form
$\hat{\mu}=re^{i\theta}$ we have, with $r>0$ fixed, 
a circle of radius
\begin{equation}\label{ratio}
\sqrt{\frac{\|Mx-{\rm rq}_{M,N}(x)Nx\|_P^2}{\|Nx\|_P^2}+
r^2}-r
\end{equation}
centred at ${\rm rq}_{M,N}(x).$ Taking the derivative with
respect to $r$ shows that this radius is a decreasing function of $r$.

If $\hat{\mu} \rightarrow 0$, then the image of the quotient function approaches the circle of radius
$\frac{\|Mx-{\rm rq}_{M,N}(x) Nx\|_P}{\|Nx\|_P}$ centred 
at ${\rm rq}_{M,N}(x)$.
Hence continuity takes place only if $x$ is an eigenvector.
Factoring $|\hat{\mu}|$ outside the square root yields that the limit at infinity is $-\hat{\mu}+\mu=
{\rm rq}_{M,N}(x)$.
\end{proof}

We denote the closure of the image of the quotient function by
$${\rm Imqf}_{M,N}(x).$$
This can be interpreted as being a Gershgorin disc. A simple formulation of this is
as follows.

\begin{corollary}\label{gers} Let $N=I$. Then  
${\rm Imqf}_{M,I}(x)$
is the Gershgorin disc 
of the first column of the Arnoldi method executed with $M$ at $x$.
\end{corollary}

The purpose of Gershgorin discs is to provide information about the location of the eigenvalues. 
Regarding ${\rm Imqf}_{M,N}(x)$, normal eigenvalue problems admit the best results in eigenvalue estimation. 
To this end, recall that $P^{-1}M^*P$ is the adjoint of a matrix $M\in \setC^{n \times n}$ 
with respect to the inner-product $(\cdot,\cdot)_P$, where $M^*$ denotes the Hermitian transpose of $M$. 
The matrix $M$ is normal if $M$ commutes with its adjoint. 


\smallskip

\begin{example}\label{eksu} 
	If $P$ can be freely chosen, then the 
	probability of having a normal matrix is one. To see this, 
assume $M\in \setC^{n \times n}$ is diagonalizable as
$M=X\Lambda X^{-1}$. If $X^*=U\hat{P}$ is the polar decomposition of
$X^*$, then $P=\hat{P}^{-2}$ yields a normalizing inner-product
for $M$, i.e., $M$ commutes with $P^{-1}M^*P$.
The task of finding $P$, which is not unique, has been addressed in \cite{HNEAR}
in the self-adjoint case.
\end{example}

\smallskip

For eigenvalue problems \eqref{gen}, normality with respect to a given inner-product is defined as follows. (For $P=I$, see \cite{bib5,IKRA,HOOKOO}.)

\begin{definition}\label{norsu} 
Assume $M- \mu N$ is nonsingular for some non-zero $\mu\in \setC$.
The eigenvalue problem \eqref{gen} is normal with respect to the 
inner-product $(\cdot,\cdot)_P$
if $M(M-\mu N)^{-1}$ is a normal matrix with respect to the inner-product $(\cdot,\cdot)_P$. 
\end{definition}

We are primarily concerned with self-adjoint eigenvalue problems which is a subset of normal eigenvalue problems. That is, the eigenvalue problem \eqref{gen} is said to be self-adjoint  if \eqref{sadjo} holds, i.e., 
$N^*PM$ is a Hermitian matrix \cite{HUKO,HOOKOO}. Clearly, 
for a given $P$, 
self-adjointness is much easier to check than normality.

Denote by $\Lambda(M,N)$ the spectrum of \eqref{gen}.

\begin{theorem}\label{lepo} 
Assume the eigenvalue problem \eqref{gen} is normal with respect to the inner-product $(\cdot,\cdot)_P$. 
Then for any $\mu \in \setC$ and 
$x\in \setC^n$ with $Nx\not=0$ holds
$${\rm dist}(\mu,\Lambda(M,N))^2\leq
|{\rm rq}_{M,N}(x)-\mu|^2+ 
|{\rm oq}_{M,N}(x)|^2-|{\rm rq}_{M,N}(x)|^2.$$
\end{theorem}

\begin{proof}  Let us assume that $\mu$ is not an eigenvalue since otherwise the claim is true.

Assume first that $N$ is invertible.
	Since the eigenvalue problem \eqref{gen} is normal, it follows that also 
$MN^{-1}$ is a normal matrix with respect to the inner-product $(\cdot,\cdot)_P$. Thereby
$${\rm dist}(\mu,\Lambda(M,N))\|Nx\|_P
=\frac{\|Nx\|_P}{\|(MN^{-1}-\mu I)^{-1}\|_P}\leq
\|(M-\mu N)x\|_P.$$
Denote the unit vector $\frac{Nx}{\|Nx\|_P}$ by $z$ and set $A=MN^{-1}$. Then using twice the
 Pythagorean theorem gives
$$\|(A-\mu I)z\|_P^2=\|((Az,z)_P-\mu)z\|_P^2+\|(A-(Az,z)_P I)z\|_P^2=$$
$$|(Az,z)_P-\mu|^2+\|Az\|_P^2-|(Az,z)_P|^2.$$
Now combining these two inequalities gives the claim.

If $N$ is not invertible, replace $N$ with $N+\epsilon M$ for small $\epsilon>0$ such that $N+\epsilon M$ is invertible. The claim follows by using the first part of the proof combined with continuity of the estimates in taking the limit $\epsilon \rightarrow 0$.
\end{proof}

This implies that if the eigenvalue problem is normal, then there exists a "perfect" quotient.
That is, we have an eigenvalue inclusion region as follows.

\begin{corollary}\label{inklu}
Assume the eigenvalue problem \eqref{gen} is normal 
with respect to the 
inner-product $(\cdot,\cdot)_P$. Then 
${\rm Imqf}_{M,N}(x)$ contains
an eigenvalue for any $x \in \setC^n$ with $Nx\not=0$.
\end{corollary}

\begin{proof} Choose $\mu ={\rm rq}_{M,N}(x)$ to have
$${\rm dist}({\rm rq}_{M,N}(x),\Lambda(M,N))^2\leq  
|{\rm oq}_{M,N}(x)|^2-|{\rm rq}_{M,N}(x)|^2=$$$$
\left(\frac{\|Mx-{\rm rq}_{M,N}(x) Nx\|_P}{\|Nx\|_P}\right)^2$$
proving the claim.
\end{proof}

Observe that there is no guarantee that ${\rm Imqf}_{M,N}(x)$
 contains an eigenvalue if the inner-product is far from being normalizing. 
That is, then any quotient can yield an unsatisfactory estimate and also lead to very unpredictable results if used as an input in an algorithm for computing eigenvectors. For this fenomenon, see \cite[Example 3.1]{HUKO}. 
Estimates are thus strongly dependent  on $P$ and therefore its choice is an issue that should be addressed in order to have useful estimates.

\smallskip

\begin{example} Assume $M=J+E$, where $J$ is the nilpotent 
forward shift and $N=I$. Use the standard Euclidean inner-product, i.e., 
$P=I$. For $n\geq 3$
take $x\in \setC^n$ all ones. Then 
${\rm rq}_{J,I}(x)=1-\frac{1}{n}$ and the radius is $\sqrt{\frac{1}{n}+\frac{n-1}{n^3}}$.
So now ${\rm Imqf}_{M,N}(x)$
does not contain an eigenvalue.
When $n$ grows, then this closure approaches the point 1, so that the distance to the eigenvalue $J$ approaches one. This means that estimates 
based on quotients can be catastrophic. 
Moreover, if $E$ is diagonal with distinct real entries, then eigenvalue problem is self-adjoint in an appropriate inner-product. 
However, if the norm of $\|E\|$ small, the quotient function in the Euclidean inner-product generates poor estimates for eigenvalues. For better behaviour, the inner-product should to be changed.
\end{example}

\smallskip

From Definition \ref{norsu} and Example \ref{eksu} we may deduce that
if the Kronecker canonical form is diagonal, then there exists a normalizing inner-product. (See Example \ref{pde}.) For the self-adjoint case,
see \cite{HNEAR}.
A practical way to change the Euclidean inner-product 
for a particular eigenvalue problem \eqref{gen} takes place through preconditioning from the left. 
Then the eigenvalue problem converts into
\begin{equation}\label{prer}
ZMx=\lambda ZNx,
\end{equation}
where $Z\in \setC^{n\times n}$ is the preconditioner.
Using the standard Euclidean inner-product means that the quotients will then involve
  $$(ZMx,ZNx)$$
which equals
  $$(Mx,Nx)_P$$
with the positive definite matrix $P=Z^*Z$. Thus quotients involving  $ZM$ and $ZN$ in the standard Euclidean inner-product
coincide with quotients involving $M$ and $N$ in the inner-product $(\cdot,\cdot)_P$.


\section{Midpoint rule for estimates in self-adjoint eigenvalue problems}\label{Sec3}
The bound of Theorem \ref{lepo} should be carefully interpreted since it does not imply that Rayleigh quotients give better estimates. 
After all, because of Theorem \ref{pallero}, the Rayleigh quotient 
${\rm rq}_{M,N}(x)$
can be regarded as providing an average of all reasonable quotients associated with $x$. For estimating extreme eigenvalues, taking the average does not appear very attractive.
For estimating large eigenvalues in the self-adjoint case, optimal quotients are clearly superior by the fact that
$$\left|{\rm rq}_{M,N}(x) \right| \leq 
\left| {\rm oq}_{M,N}(x) \right|.$$
In particular, if $M$ is positive definite, then in the standard eigenvalue problem
\begin{equation}\label{warren}
{\rm rq}_{M,I}(x)\leq {\rm oq}_{M,I}(x)\leq 
\lambda_n,
\end{equation}
where $\lambda_n$ denotes the largest eigenvalue of $M$.
The first inequality is equality if and only if $x$ is an eigenvector. 
However, when $x$ has been randomly picked,
the first inequality can be expected to be much like the difference between the arithmetic mean versus the quadratic mean. (That is, as a model, assume $M$ is diagonal and take
$x$ to be all ones multiplied by $\frac{1}{\sqrt{n}}$.)  
So the gap can be notable. 
Empirically the consequences of this can be seen in numerical experiments
\cite[Example 3.1.]{HUKO}. 

These estimates can be improved and generalized to apply to 
any part of the spectrum
by carefully inspecting the quotient function \eqref{funa}. 
The aim is at cleverly choosing $\mu$.
In the self-adjoint case
the object of interest is $\setR$ only.
Based on its growth properties, this allows formulating simple guidelines on how to use the quotient function in estimating eigenvalues.

\begin{theorem}\label{blaa}
  Assume the eigenvalue problem \eqref{gen}   is self-adjoint with respect to the 
inner-product $(\cdot,\cdot)_P$. 
  If $x\in\setC^n$ is not an eigenvector, 
   then \eqref{funa} is an increasing function on $\setR$
discontinuous only at ${\rm rq}_{M,N}(x)$ with
$$\lim_{\mu \rightarrow {\rm rq}_{M,N}(x)^{\pm}}
{\rm oq}_{M-\mu N,N}(x)+\mu= 
{\rm rq}_{M,N}(x) \mp\|Mx-{\rm rq}_{M,N}(x)Nx\|_P$$
and $\lim_{\mu \rightarrow \pm \infty}{\rm oq}_{M-\mu N,N}(x)+\mu
={\rm rq}_{M,N}(x)$.
\end{theorem}


Regardless of the problem being self-adjoint or not,
the quotient function is constant on $\setC$ if and only if $x$ is an eigenvector. (The respective "Rayleigh quotient function" ${\rm rq}_{M-\mu N}(x) +\mu$ is constant for any $x$.) Moreover, 
the quotient function is discontinuous at $\mu ={\rm rq}_{M,N}(x)$ if and only if $x$ is not an eigenvector. The gap of this discontinuity provides a measure how far $x$ is from being an eigenvector.


Whenever $x$ is a good eigenvector approximation, it is 
possible to recover the respective extreme eigenvalue exactly with a  unique choice of $\mu$; see Corollary \ref{inklu} and Figure \ref{fig:quotifu}. Readily finding this "perfect" quotient  is not realistic, though.
To inexpensively choose $\mu$ in a reasonable way, for the left end we have the following range for the parameter $\mu$ for
the quotient function to yield better estimates than the Rayleigh quotient.
(The right end is treated similarly.)
Denote the eigenvalues 
by $\lambda_1\leq \lambda_2\leq \cdots \leq \lambda_n$.

\begin{corollary}\label{siirto}
Assume $N$ is invertible and let $\mu> {\rm rq}_{M,N}(x)$ 
be such that 
$\left| \lambda_1 -\mu\right| \geq \left|\lambda_n-\mu \right|.$
Then 
$$\left|\lambda_1-\left({\rm oq}_{M-\mu N,N}(x)+\mu\right)\right|\leq\left|\lambda_1-{\rm rq}_{M,N}(x)\right|$$
with equality holding if and only if $x$ is an eigenvector.
\end{corollary}

\begin{proof} Again, we have
$\frac{\| (M-\mu N)x\|_P}{\|Nx\|_P}= 
\frac{\|(A-\mu I)z\|_P}{\|z\|_P}$ with self-adjoint $A=MN^{-1}$ and $z=Nx$.
Because of the assumption, the norm of $A-\mu I$ is $|\lambda_1-\mu|$. Consequently,
$$|\lambda_1- \mu |\geq \frac{\|(A-\mu I)z\|_P}{\|z\|_P} \geq
\frac{((A-\mu I)z,z)_P}{(z,z)_P}$$
holds. The last inequality is equality if and only if $z$ is an eigenvector.
\end{proof}

We know that the quotient function is increasing.
This means that, to estimate the smallest eigenvalue, it is a good choice to take 
$\mu> {\rm rq}_{M,N}(x)$ such that roughly
\begin{equation}\label{optikko}
\mu \, \mbox{ is the midpoint of the spectrum. }
\end{equation}
Thus, some information about the extreme eigenvalues
is required here. One option is to take a few Gaussian random vectors, compute their Rayleigh quotients and take $\mu$ to be the average of their minimum and maximum. In the standard Hermitian eigenvalue problem, such information can be generated by taking
a few steps of the Hermitian Lanczos method. With this additional information, assuming $N$ is invertible, 
the optimal quotient iteration \cite{HUKO} reads as Algorithm \ref{alg:factspar}. 
The  purpose of line 4 is to check of how near $Aq$ and $Bq$ are being linearly independent.
Step 7 is the most time consuming part,  requiring solving
linear systems. If done iteratively, observe that 
\begin{equation}\label{hermisy}
B^*P(A-lB)
\end{equation}
 is a Hermitian matrix, although it may not be wise to explixitely compute it. 
So, if the linear system is preconditioned with $B^*P$ from the left, 
this should be taken into account in choosing the iterative solver.

\begin{algorithm}[t]
   \caption{Optimal quotient iteration for an eigenvector approximation associated with extreme eigenvalues of a self-adjoint eigenvalue problem}
   \label{alg:factspar} 
   \begin{algorithmic}[1]
     \State Read $n$-by-$n$ matrices $M$ and $N$ and an approximate unit 
      eigenvector $x$ and a tolerance $\epsilon$
      \State Read an approximation to the midpoint $\mu$ of the spectrum of \eqref{gen}
      \State Set $A=M-\mu N$ and $B=N$ 
     \While{ $\sigma_2([Ax\;Bx])>\epsilon$ }  
      \State Compute $w_1=\frac{Ax}{\|Ax\|_P}$ and  $w_2=\frac{Bx}{\|Bx\|_P}$
     \State Set $z=\frac{1}{\sqrt{2+2|(w_2,w_1)_P|}}(\frac{(w_2,w_1)_P}{
|(w_2,w_1)_P|}w_1+w_2)$
     \State Compute $l=\frac{(Ax,Bx)_P}{|(Ax,Bx)_P|}
\frac{\|Ax\|_P}{\|Bx\|_P}$    
     \State Solve $(A-lB)\hat{x}=z$ and set $x=\hat{x}/\|\hat{x}\|_P$
           \EndWhile
        \end{algorithmic}
\end{algorithm}

\begin{definition} A self-adjoint eigenvalue problem \eqref{gen} is said to be positive semi-definite if its eigenvalues are non-negative.
\end{definition}

Of course, the problem can be that of recovering whether an eigenvalue problem is positive definite. This takes place
in optimization when classifying a critical point; one needs to compute the smallest eigenvalue to check whether
it is negative or not.
Still, many problems, e.g., in structural engineering are a priori known to be 
positive definite.
When estimating the largest eigenvalue in the positive semi-definite case we can also use either
the Rayleight or optimal quotient in roughly estimating the midpoint.
(We assume $N$ to be invertible for $\lambda_n$ to be finite.)
The task of estimating the largest eigenvalue arises, for example, in estimating the largest singular value of
a very large matrix. We have improved estimates as follows.

\begin{corollary}\label{pual} 
Assume \eqref{gen} is positive semi-definite with $N$ invertible.  
Then both with $\mu =\frac{1}{2}{\rm rq}_{M,N}(x)$ 
and $\mu =\frac{1}{2}{\rm oq}_{M,N}(x)$ 
holds
$$0\leq \lambda_n-\left({\rm oq}_{M-\mu N,N}(x)+\mu\right)
\leq 
\lambda_n -{\rm oq}_{M,N}(x)
\leq\lambda_n-{\rm rq}_{M,N}(x)$$
with equalities holding if and only if $x$ is an eigenvector.
\end{corollary}

\begin{proof} The third inequality holds by the fact that
${\rm oq}_{M,N}(x)$ and ${\rm rq}_{M,N}(x)$
are underestimates to $\lambda_n$ and 
$0\leq {\rm rq}_{M,N}(x) \leq {\rm oq}_{M,N}(x)$.
By the Cauchy-Schwarz inequality, equality holds if and
and only if $x$ is an eigenvector. 

The second inequality holds since the quotient function is increasing. Equality holds if and only if $x$ is an eigenvector. 

The first inequality holds by the fact that
the eigenvalues of $(M-\mu N)z=\lambda N z$ coincide with the eigenvalues of the matrix
$(M-\mu N)N^{-1}$. Its norm is $\lambda_n-\mu$. Thereby 
${\rm oq}_{M-\mu N,N}(x)$ is an underestimate to $\lambda_n-\mu$. Again, equality holds if and only if $x$ is an eigenvector.
\end{proof}

In practice this means that first an auxiliary quotient is generated for the midpoint estimation  which is then used to compute the actual quotient used to approximate $\lambda_n$. Observe that this again yields us an improved estimate 
$ \frac{1}{2}\left({\rm oq}_{M-\mu N,N}(x)+\mu\right)$
to the midpoint of the spectrum. So the construction  can be iteratively repeated; see Algorithm \ref{alg:pyor}. 

\begin{algorithm}[t]
   \caption{Quotient for the largest eigenvalue of a self-adjoint positive semidefinite problem}
   \label{alg:pyor} 
   \begin{algorithmic}[1]
     \State Read $n$-by-$n$ matrices $M$ and $N$ and an approximate unit 
      eigenvector $x$
      \State Compute $v=Mx$, $w=Nx$ and $l=\|Nx\|_P$
      \State Set $\mu=\frac{\|v\|_P}{2l}$
      \For{\texttt{until convergence}} 
\State  Compute $\alpha=\frac{\|v-\mu w\|_P}{l}+\mu$   
      \State Set $\mu=\frac{\alpha}{2}$ 
        \EndFor   
        \end{algorithmic}
\end{algorithm}
 
\begin{corollary}\label{paras}  Assume \eqref{gen} is positive semi-definite with $N$ invertible  
and $\mu =\frac{1}{2}{\rm oq}_{M,N}(x)$.
Then  $$0\leq \lambda_n-\alpha \leq
\lambda_n-\left({\rm oq}_{M-\mu N,N}(x)+\mu\right),
$$
where $\alpha$ is computed by Algorithm \ref{alg:pyor}.
\end{corollary} 
 
 Clearly, the quality of the midpoint estimation
with Algorithm \ref{alg:pyor}
depends on how close the left end of the spectrum is to the origin.
This is illustrated in Examples \ref{kungfood} and \ref{ubo} below. 

\smallskip

\begin{example}\label{kungfood} For this tiny but educative standard Hermitian eigenvalue problem, see \cite[Example 27.1]{TRB}. It was also treated in \cite[Example 3.2]{HUKO} when comparing
the quotient iterations  with $\mu =0$ and $\mu =\infty$. That is, 
$$M=\left[ \begin{array}{ccc}
2&1&1\\ 1&3&1\\
1&1&4
\end{array} \right]$$
and $N=I$.  We have $P=I$.
Gershgorin discs reveal that $M$ is positive semi-definite.
The used vector 
$x=\left[ \begin{array}{ccc}
1&1&1
\end{array} \right]^T/\sqrt{3}$ is aimed at approximating the largest eigenvalue.  This means taking $\mu < {\rm rq}_{M,I}(x)=5$. 
There is a unique $\mu$ giving the largest eigenvalue; see
Figure \ref{fig:quotifu}. 
Since $M$ is positive semi-definite, take either $\frac{1}{2}{\rm rq}_{M,I}(x)=2.5$ 
or $\frac{1}{2}{\rm oq}_{M,I}(x)=2.533$ to approximate the midpoint; see
Corollary \ref{pual}. Then we obtain the estimate 
${\rm oq}_{M- 2.533I,I}(x)+2.533=5.1316$.
Algorithm  \ref{alg:pyor} improves this by giving an estimate $5.1333$.
These should be compared against $\Lambda(M,I)=\{1.3249,\,  2.4608,\, 5.2143\}$ rounded to $5$ digits.
\end{example}

\smallskip 

\begin{figure}[t]
\centering
\caption{The graph of the quotient function  
of Example \ref{kungfood}.
The eigenvalues of $M$ are depicted vertically with 'x' on the $y$-axis. The Rayleigh quotient is $5$. With Algorithm 1 we attain $5.1333$.
The 
value of the quotient function with $\mu$
 being the exact midpoint \eqref{optikko} is $5.183$ and is
 depicted with 'o'. }
\label{fig:quotifu}
\end{figure}




A very common problem is to estimate the smallest eigenvalue of a very large positive definite eigenvalue problem. There are two options to apply Algorithm \ref{alg:pyor}.
The simplest option,  requiring  no additional information,  
is to proceed similarly after interchanging the roles of $M$ and $N$. For this, see \eqref{sjwe} below.  The second option is to 
consider
\begin{equation}\label{ekava}
-(M-rN)x=\lambda Nx,
\end{equation}
where $r>0$ is such that the largest eigenvalue of this reformulation corresponds
to the nearest eigenvalue of the original eigenvalue problem to the origin.  
Then some additional information is required in choosing $r$.

Suppose now that an eigenvalue of a self-adjoint eigenvalue problem \eqref{gen}, which is not necessarily extreme, is to be estimated near a given point $\zeta \in \setR$. 
(For problems of finding eigenvalues inside a gap in quantum physics, see 
\cite{HIKR} and, in particular, \cite[Chapter 2]{ELS} and \cite{ES} for a concise review
of approaches.) 
To this end, consider the reformulated eigenvalue problem
\begin{equation}\label{lahi}
Nx=\lambda (M-\zeta N)x.
\end{equation} 
Now the situation is about estimating the largest eigenvalues in absolute value of this 
reformulated eigenvalue problem since
\begin{equation}\label{sjwe}
\lambda\, \mbox{ is an eigenvalue of \eqref{lahi} if and only if}\; 
\frac{1}{\lambda}+\zeta\, \mbox{ is an 
eigenvalue of \eqref{gen}}.
\end{equation}
Thus, with  $\lambda$ as large as possible in absolute value we obtain estimates as near to $\zeta$ as possible
for the original eigenvalue problem \eqref{gen}.
  This converts into applying Corollary \ref{siirto} for producing good estimates. That is, now a midpoint estimate 
$\mu$ must be produced for the spectrum of \eqref{lahi}.
  Once $\mu$ has been set, one can execute the corresponding optimal quotient iteration; see Algorithm  \ref{alg:factspar}.
Thus, computationally our approach makes no distinction between estimating extreme and interior eigenvalues.

To estimate the smallest positive eigenvalue of a positive definite problem, one simply sets $\zeta=0$ 
in \eqref{lahi}
if no additional information is available. This means interchanging the roles of $M$ and $N$.
This reformulation can be used with Corollary \ref{pual}  and \ref{paras} 
in the unbounded case as well; see Appendix B.
 For an illustration, suppose $M$ is a  self-adjoint unbounded positive definite operator 
on a Hilbert space 
and one is interested in estimating its smallest eigenvalues. 
This is a classical and  very important problem \cite{KATO}.
Besides \eqref{sqr}, there are three other ways to produce estimates.
Reformulate this standard eigenvalue problem such that \eqref{lahi} reads
\begin{equation}\label{ree}
x=\lambda Mx,
\end{equation}
so that that one is interested in the largest eigenvalues.
Assume $x$ is a trial vector. 
Now 
${\rm rq}_{I,M}(x)=\frac{(x,Mx)}{\|Mx\|^2}$ 
which is suggested in \cite{HIKR}. Thus, by \eqref{sjwe} we obtain
$\frac{1}{{\rm rq}_{I,M}(x)}=\frac{\|Mx\|^2}{(x,Mx)}$ to approximate the
smallest eigenvalue  of $M$.
But this is not the best estimate in terms of $x$. We can take $\frac{1}{{\rm oq}_{I,M}(x)}$.
To improve this, 
use the readily available estimate 
$\mu =\frac{1}{2}{\rm oq}_{I,M}(x)$
for the midpoint. Then  
${\rm oq}_{I-\mu M,M}(x)+\mu$ gives a better estimate to the largest eigenvalue of \eqref{ree}; see Corollary \ref{pual}.  
To improve this further, execute  Algorithm \ref{alg:pyor}. 
Its reciprocal then gives an upper bound on the
the smallest eigenvalue of $M$. The respective quotient iteration is Algorithm \ref{alg:factsparpo}.

Let us consider a simple benchmark problem to compare these four estimates.

\smallskip

\begin{example}\label{ubo} 
Let $M=-\frac{d^2}{dt^2}$ on $L^2(0,1)$ for the 
 eigenvalue problem $Mx=\lambda x$
with the boundary conditions $x(0)=x(1)=0$. The smallest eigenvalue is known, being  
$\pi^2\approx 9.8696$. Take the trial approximate eigenvector  $x(t)=\sqrt{30}t(1-t)$. Then we have
${\rm rq}_{M,I}(x)=10$, $\frac{1}{{\rm rq}_{I,M}(x)}=12$ and
$\frac{1}{{\rm oq}_{I,M}(x)}=2 \sqrt{30}\approx 10.9545$.
Executing  Algorithm \ref{alg:pyor} gives the estimate 10 which, intriguingly, is the same as ${\rm rq}_{M,I}(x)$.
So it seems 10 is best what can be produced in terms of quotients using
this particular approximate eigenvector. 
\end{example}

\smallskip

\begin{algorithm}[t]
   \caption{Optimal quotient iteration for an eigenvector approximation associated with the smallest eigenvalue of
a positive definite self-adjoint eigenvalue problem}
   \label{alg:factsparpo} 
   \begin{algorithmic}[1]
     \State Read $n$-by-$n$ matrices $M$ and $N$ and an approximate unit 
      eigenvector $x$ and a tolerance $\epsilon$
     \While{ $\sigma_2([Mx\;Nx])>\epsilon$ }  
      \State Compute $w_1=\frac{Mx}{\|Mx\|_P}$ and  $w_2=\frac{Nx}{\|Nx\|_P}$
     \State Set $z=\frac{1}{\sqrt{2+2|(w_2,w_1)_P|}}(\frac{(w_2,w_1)_P}{
|(w_2,w_1)_P|}w_1+w_2)$
     \State Execute Algorithm \ref{alg:pyor} for the problem $Nw=\lambda Mw$ and set $l=1/\alpha$
     \State Solve $(M-lN)\hat{x}=z$ and set $x=\hat{x}/\|\hat{x}\|_P$
           \EndWhile
        \end{algorithmic}
\end{algorithm}

\section{Preconditioning and  variational principles for ge\-nerating
approximate eigenvectors}\label{Sec4}
All the preceding estimates with quotients require high quality approximate eigenvectors. 
First, to quickly attain a cubic speed of convergence and minimize the number of iterations, the starting vector needs to be good.
Second, in executing a quotient iteration such as Algorithm \ref{alg:factspar},
one must be aware that it is strongly dependent on the starting vector; see \cite[pp. 84--85]{PA} for the Rayleigh quotient iteration. 
That is, unlike in executing the power method in the standard 
eigenvalue problem for the largest eigenvalue in modulus,  a randomly chosen starting vector typically yields unpredictable convergence results. 
For an illustration of this effect, see \cite[Example 3.1.]{HUKO}. In fact, generating a starting vector
 to attain converge towards eigenvalues of interest
is arguably the most 
challenging part of the approximation process.
In the self-adjoint case we can overcome this by resorting to variational principles.

Assume the task is to compute  an approximate eigenvector associated with an eigenvalue near $\mu\in \setR$.
To this end, consider the quotient function minus $\mu$.
Then, in a self-adjoint (or normal) eigenvalue problem \eqref{gen}, let us vary $x$ by inspecting 
\begin{equation}\label{minni}
\min_{x \in \setC^n}\frac{\|(M- \mu N)x\|_P^2}{\|Nx\|_P^2}=
\min_{x \in \setC^n}
\frac{((M-\mu N)^*P(M- \mu N)x,x)}{(N^*PNx,x)}
\end{equation}
If $N$ is invertible, then this is equivalent to 
$\min_{y \in \setC^n}\frac{\|(MN^{-1}- \mu I)y\|_P^2}{\|y\|_P^2}$. 
Since $MN^{-1}$ is self-adjoint, the minimum is realized at an eigenvector corresponding to an eigenvalue nearest to $\mu$. 
And, if $N$ is singular, the same argument applies by considering 
$$\max_{x \in \setC^n}\frac{\|Nx\|_P^2}{\|(M- \mu N)x\|_P^2}$$
assuming $\mu$ is not an eigenvalue.
To approximate this eigenvector, let us derive a descent step.
Fix a starting point $x\in \setC^n$. Applying the conjugate co-gradient to \eqref{minni} gives the direction
\begin{equation}\label{des}
(M-\mu N)^*P(M-\mu N)x-\frac{\|(M- \mu N)x\|_P^2}{\|Nx\|_P^2}
N^*PNx
\end{equation}
to descend from $x$. (For taking the conjugate co-gradient
for descent, see
\cite{BRA}.)
Denote by $q_1=\frac{x}{\|x\|}$ and
$q_2$ the descent direction orthonormalized against $q_1$. 
Let $V=\smat{q_1&q_2}$.
Then
\begin{equation}\label{minnimi}
\min_{v \in \setC^2}\frac{\|(M- \mu N)Vv\|_P^2}{\|NVv\|_P^2} 
\end{equation}
approximates \eqref{minni} and
can be solved in terms of the small eigenvalue problem 
\begin{equation}\label{apupa}
V^*(M-\mu N)^*P(M-\mu N)Vv=\lambda 
V^*N^*PNVv
\end{equation}
by finding an eigenvector $v$ corresponding to its smallest eigenvalue.
This gives a new starting point $x=Vv$ to descend from into the direction  \eqref{des}. Strictly interpreted, this is an improved descent step by the fact that the new starting point is an optimal linear combination of the previous starting point $x$ and \eqref{des}. At this point also $\mu$ can be up-dated, by computing an appropriate quotient using this new starting point $x$.
 
Based on the convergence of the power method, a problem with the descent direction \eqref{des} is that
multiplications with $(M-\mu N)^*P(M-\mu N)$ will emphasize 
directions associated with large eigenvalues. To reduce this effect, 
a more rapid descent can be expected to require 
preconditioning. To this end, take any invertible matrix $Z\in\setC^{n \times n}$ and consider
\begin{equation}\label{pre}
(M-\mu N)Zx=\lambda NZx
\end{equation}
by the fact that this equivalent eigenvalue problem is self-adjoint if and only if  \eqref{gen} is. That is,
$Z^*N^*PMZ$  is a Hermitian matrix if and only if $N^*PM$ is.
As self-adjointness is preserved 
in any left preconditioninng, this is a huge relaxation compared with what is usually suggested \cite[p. 108]{AK} 
Therefore also variational principles 
can be applied to this equivalent eigenvalue problem; see Algorithm \ref{alg:dese}.
A natural option is to take $Z$ 
to be an approximation to $(M-\mu N)^{-1}$.
To see its effect, consider the following model case, i.e., the standard Hermitian eigenvalue problem.

\smallskip 

\begin{example}\label{lap} Let $N=I$ and $\mu=0$, so that the task is to estimate eigenvalues of a Hermitian matrix $M$ near the origin. We may take
$P=I$.	
If $Z=M^{-1}$, then preconditioning transforms
the eigenvalue problem into $x=\lambda M^{-1}x$. For this problem, 
\eqref{minni}
reads $\min_{x \in \setC^n}\frac{\|x\|^2}{\|M^{-1}x\|^2}$. Taking now the descent direction, 
the column space of $V$ is ${\rm span}\{x,M^{-2}x\}$, containing vectors
power iterated with $M^{-1}$. 
This has the desired effect of emphasizing eigenvectors corresponding
to eigenvalues near the origin.
In practice $Z$ is an approximation to $M^{-1}$, so that  we are transforming the standard eigenvalue problem $Mx=\lambda x$ 
into a generalized eigenvalue problem $MZx=\lambda Zx$. 
Then \eqref{minni} reads
$\min_{x \in \setC^n}\frac{\|MZx\|^2}{\|Zx\|^2}.$
\end{example}

\smallskip

In the preceding example, if $M$ is indefinite, i.e., an interior eigenvalue is being searched, then indefinite preconditioning techniques for generating $Z$ need to be invoked. 
In particular, $Z$ can be a very rough estimate of the inverse since
\eqref{pre} is in any case equivalent to the original eigenvalue problem. 
Observe that this precondition strategy transforms
a standard eigenvalue problem into a generalized eigenvalue problem by the fact that now no distinction is made between these problems. This notably differs from the usual approach; see \cite[Section 8.3]{SAAD}.

\begin{algorithm}[t]
   \caption{Preconditioned descent method for
   approximating an eigenvector}
   \label{alg:dese} 
   \begin{algorithmic}[1]
     \State Read $n$-by-$n$ matrices $M$ and $N$, approximate unit 
      eigenvector $x$, point $\mu\in \setR$ and preconditioner $Z$
      \State Denote $\hat{M}=(M- \mu N)Z$ and $\hat{N}=NZ$
      \For{\texttt{until convergence}}      
      \State Orthogonalize $\hat{M}^*P\hat{M}x-
      \frac{\|\hat{M}x\|_P^2}{\|\hat{N}x\|_P^2}\hat{N}^*P\hat{N}x$ against $x$ to have $\hat{x}$ 
\State  Set $V=\smat{x&\hat{x}}$   
      \State Solve $V^*\hat{M}^*P\hat{M}Vv=\lambda 
V^*\hat{N}^*P\hat{N}Vv$ for the eigevector $v$ corresponding to the
smallest eigevalue
\State Set $x=Vv$
        \EndFor
        \State Set $x:=Zx$   
        \end{algorithmic}
\end{algorithm}
 
Algorithm \ref{alg:dese} requires the least amount of storage. If we collect all the generated vectors, after orthogonalization, into 
$V_k=\smat{q_1&q_2&\cdots&q_k}$, then we are concerned with
\begin{equation}\label{minnimimo}
\min_{v \in \setC^k}\frac{\|(M- \mu N)Vv\|_P^2}{\|NVv\|_P^2}. 
\end{equation}
This can be solved accordingly, with increased storage requirements though.
This can be classified as a preconditioned folded spectrum method \cite{HUHNEV} for self-adjoint generalized eigenvalue problems . We have the following for Example \ref{lap}.

\begin{theorem} Assume $M$ is invertible and $N=I$.  If $\mu=0$ and $Z=M^{-1}$, then 
the column space  of $V=\smat{q_1&q_2&\cdots&q_k}$ generated by the descend method equals 
 $${\rm span}\{x,M^{-2}x,M^{-4}x, \ldots, M^{-2k}x \}.$$
\end{theorem}

\begin{proof}
For a descent direction, at each step $j$ we have to solve
an eigenvalue problem
$$V_j^*(I-\lambda M^{-2})V_j=0$$
for the smallest eigenvalue. The corresponding eigenvector cannot be a linear combination of $j-1$ columns of $V_j$ by the fact that at the 
$(j-1)$th step  the conjugate cogradient was nonzero. 
\end{proof}




\section{Estimating several eigenvalues}\label{Sec5}
In practice there are, typically, two types of  large scale eigenvalue problems.
One consists of finding a small number of eigenvalues and eigenvectors 
from the left (or right) end of the spectrum. The other is that of finding
a small number of eigenvalues and eigenvectors inside a gap. These problems require having tools to compute eigenvectors one by one and simultaneously avoid repeated convergence to an eigenvalue already computed.
In the self-adjoint the case, orthogonality of the associated eigenvectors can be used to achieve this.

\begin{theorem}
  Assume \eqref{gen} is self-adjoint and $\mu \in \setR$. If $x_1$ and $x_2$ are
  two eigenvectors associated with different eigenvalues,
  then
\begin{equation}\label{ehtosa}
  ((M-\mu N)x_1,(M-\mu N)x_2)_P=0.
  \end{equation}
\end{theorem}

\begin{proof}
For $j=1,2$, consider
 $(M-\mu N)x_j=\lambda_j Nx_j.$
This gives, after taking the inner-product with $(M-\mu N)x_k$ with
$k=2,1$, 
$$((M-\mu N)x_1,(M-\mu N)x_2)_P=\lambda_1( Nx_1,(M- \mu N)x_2)_P$$
and 
$$((M-\mu N)x_2,(M-\mu N)x_1)_P=\lambda_2( Nx_2,(M- \mu N)x_1)_P.$$
Since $(M- \mu N)^*PN$ is a Hermitian matrix, we have  
$( Nx_1,(M- \mu N)x_2)_P=( Nx_2,(M- \mu N)x_1)_P$.
This gives, after subtracting, the claim as $\lambda_1 \not=\lambda_2$.
\end{proof}

For computations this means the following.
Suppose one eigenvalue and an associated 
eigenvector $x_1$ has been found. To look for another, an approximate
eigenvector $x$ aimed at finding another eigenvalue  should be taken to
satisfy the orthogonality condition \eqref{ehtosa}.
This means replacing $x$ with 
\begin{equation}\label{ortoehto}
x-\frac{((M- \mu N)x,(M-\mu N)x_1)_P}{\| (M- \mu N)x_1\|_P^2} x_1
\end{equation}
during the computational process.
This can be repeated, i.e.,
always imposing this orthogonality condition
against the eigenvectors found so far allows finding eigenvalues one by one.

\section{Numerical experiments}\label{Sec6}
Next two numerical experiments 
are performed. The task is to 
approximate a few  eigenpairs of a self-adjoint eigenvalue problem 
$$Mx=\lambda Nx$$
near a given $\mu\in \setR$.  Both of the experiments are realistic and challenging.

To sum up, the steps required are as follows.\\
1. {\bf Preparatory steps:}\\
 $ \mbox{ }\;  \bullet$ Generate matrices $M$ and $N$.\\
 $ \mbox{ }\; \bullet$ Construct a self-adjoining inner product $(\cdot,\cdot)_P$ and a rapid algorithm to apply $P$.\\
$ \mbox{ }\;  \bullet$ Choose either \eqref{gen} or 
\eqref{lahi} so as to approximate an extreme eigenvalue.\\
2. {\bf Variational steps for an approximate eigenvector:}\\
$ \mbox{ }\;  \bullet$ Build a preconditioner $Z$ approximating $(M-\mu N)^{-1}$.\\
$ \mbox{ }\;  \bullet$ Execute Algorithm \ref{alg:dese}.\\
3. {\bf Execution of an optimal quotient iteration:}\\
$ \mbox{ }\;  \bullet$ Provide an estimate for the midpoint of the spectrum.\\
$ \mbox{ }\;  \bullet$ Execute Algorithm \ref{alg:factspar} or \ref{alg:factsparpo}.

Some remarks are in order. In Step 1, there are degrees of freedom in choosing $P$. The choice affects the convergence. Step 2 is absolutely critical for a correct convergence. Step 3 may not be needed if Step 2 yields sufficiently good approximations.  This is not to be expected though. That is, the steepest descent method  convergences only linearly while the speed of convergence of quotient iteration is cubic.
The aim is that two or at most tree iterations with  Algorithm \ref{alg:factspar}  or \ref{alg:factsparpo}
are required
\smallskip

The computations were executed on Lenovo Thinkpad X13 Yoga laptop with 13th Gen Intel(R) Core(TM) i5-1335U processor and 32 GB of RAM, using Matlab version R2024a with Partial Differential Equation Toolbox version 24.1  
\smallskip

\begin{example}\label{ekaesim} 
	This benchmark\footnote{Curiously, all the experiments we have found in the litterature treat very simplified versions of this eigenvalue problem, without actually solving the original problem described in Matrix Market.} eigenvalue problem 
is a discretization of a fluid flow generalized eigenvalue problem of a dynamic analysis in structural engineering. The matrices are downloadable from the Matrix Market \cite{MATMA} with $M$ being BCSSTK13, a positive definite matrix, while $N$ is BCSSTM13, a positive semi-definite matrix. Both are sparse such that $n=2003$ is the dimension of the problem. The task is to compute the smallest eigenvalue. What makes this problem tough is the conditioning of $M$, being about $4.6*10^{10}$, and the singularity of $N$.
 
 For the inner-product, since $N$ is not invertible, we take a 
linear combination of $M$ and $N$. For the sparsest possible option, we set $P=M^{-1}$. 
To perform operations with $M^{-1}$,  a sparse Cholesky factorization with reordering implemented in Matlab is applied. 
	
	
From the matrices we may conclude that the eigenvalue problem is positive definite,  i.e., the eigenvalues are on the positive real axis, so that the task is to estimate the eigenvalue nearest to the origin. (We have no knowledge of how near to the origin the eigenvalue is.)
Being an extreme eigenvalue at the left end of the spectrum, we  choose the original formulation \eqref{gen}.

We take a random initial guess as a starting vector. With this, three iterations with Algorithm \ref{alg:dese},
using $\mu=0$ and $Z=M^{-1}$ as a preconditioner (available with no extra cost), gives an approximate eigenvalue 148.66.
(So at this point we know that the smallest eigenvalue appears to be quite large.) 
Switching to Algorithm  
\ref{alg:factsparpo} 
then gives $\lambda_1\approx 147.5340745961005$ in three iterations. See the left panel of Figure
\ref{bcsst}.

The convergence and quality of approximate eigenvectors $x$ can be assessed in terms of  the loss of linear dependency of the vectors $w_1=\frac{Mx}{\|Mx\|_P}$ and $w_2=\frac{Nx}{\|Nx\|_P}$  by monitoring the 2nd singular value  $\sigma_2([w_1\ w_2])$. 
This is depicted in the right panel of Figure \ref{bcsst}. 
For the quotient iteration the final value for this is $1.660 \cdot 10^{-11}$ while for the eigenvector produced by Matlab’s eigs-function it is  $1.2947 \cdot 10^{-9}$. So our iterations yield more accurate results. 
\end{example} 

\begin{figure}
	\centering
	\includegraphics[width=99.5mm, height=39.5mm]{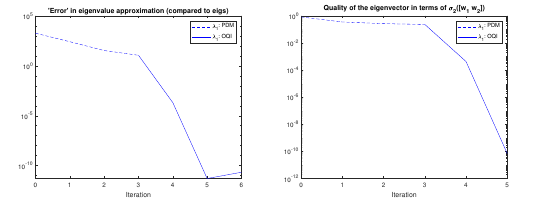}
	\caption{Convergence of the eigenpairs for the fluid flow generalized eigenvalue problem of Example \ref{ekaesim}.  First thee itrations with Algorithm \ref{alg:dese} are executed. Thereafter two iterations with Algorithm  
\ref{alg:factsparpo} 
 suffices for convergence.
		The left panel shows the difference between the eigenvalue approximation and 
the  eigenvalue computed with Matlab's eigs command. The right panel displays the respective loss of linear dependency of the vectors $w_1$ and $w_2$.}
	\label{bcsst}
\end{figure}
\smallskip

\begin{example}\label{wawee}
	This experiment is concerned with a waveguide problem treated
	in \cite{NRU}.  The Z-shaped waveguide consists of two arms of width 1 and length $R$ which are connected in right angles (see Figure \ref{Fig: Image}). Parameters $H$ and $L$ describe the horizontal and vertical dimensions of the box containing the junction of the two arms, correspondingly. 
	
	Discretized by using FEM, the matrix $M$ is Hermitian positive definite resulting
	from discretizing the Laplacian and $N$ is the mass matrix of the FEM basis used. the dimension of the problem is $n = 57509$. We are interested in finding the discrete spectrum, corresponding to the bound states, which is known to be located on the  $[0,\pi^2]$. With the chosen parameters $L=3$ and $R=15$, the discrete spectrum is known to consist of two eigenvalues \cite{NRU}. The task is to compute them. This
	problem is very tough since the second eigenvalue $\lambda_2$ is very close to the continuous
	spectrum.
	
	 Again we take $P = M^{-1}$ to have a self-adjoining inner-product. We use the formulation \eqref{gen}.
	 
	First we compute the smallest eigenvalue $\lambda_1$. Using a random starting vector  we take four iterations with Algorithm 4, using $\mu = 0$ and $Z=M^{-1}$ as a preconditioner. The sparse Cholesky factorization with reordering
	was again used for applying $M^{-1}$.  Thereafter two
	quotient iterations are needed with Algorithm  
\ref{alg:factsparpo} 
for an eigenvalue $\lambda_1=8.896137035724147$ in near full machine
	accuracy. In Figure \ref{Fig: Image} we display
	the  corresponding eigenvector.
	
	To compute the second eigenvalue $\lambda_2$,  the line 7 of Algorithm \ref{alg:dese} must be augmented
	with the orthogonality condition \eqref{ehtosa} for the approximate eigenvector $x$ to be against
	the computed eigenvector $x_1$. Its purpose is to steer the iteration away from the
	already computed eigenvalue $\lambda_1$. Since $\lambda_2$ is very close to the (discretized) continuous
	spectrum, preconditioning must be carefully devised so that switching Algorithm \ref{alg:dese}  to
	a quotient iteration results in correct convergence. The preconditioner was taken to
	be $Z=M-\mu N$ with $\mu =\lambda_1 +\epsilon$. In this example $\epsilon=10^{-7}$ was used. Since $Z$ is no longer positive definite, partially pivoted LU-decomposition with reordering is required. With this choise, four iterations with Algorithm \ref{alg:dese} are enough to ensure that Algorithm \ref{alg:factspar}  finds the correct eigenvalues with two iterations. For this last stage, $\mu=27000$ was used to approximate the mid-point of the spectrum and  $P=N^{-1}$ was used for the self-adjoining  inner-product.  
	In Figure \ref{conver}  the converge of the eigenpairs is displayed. 
	
	Let us emphasize that to be sure, with hight probability, that $\lambda_2$ gets correctly
	computed instead of computing points of the (discretized) continuous spectrum, a higher
	number of iterations with Algorithm \ref{alg:dese} were taken, although four did suffice. That is, the construction of the starting vector for a quotient iteration is a very delicate and highly critical issue for
	getting correct results.

\end{example}

\begin{figure}
\centering
\includegraphics[width=99.5mm, height=39.5mm]{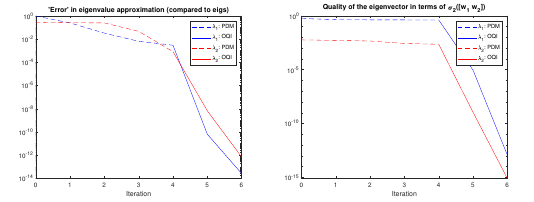}
\caption{Convergence of the eigenpairs for the waveguide problem of Example \ref{wawee}.  
First four itrations with Algorithm \ref{alg:dese} are executed. Thereafter
two iterations with Algorithm \ref{alg:factspar} are needed for  $\sigma_2([w_1,w_2])<10^{-10}$.
The left panel shows the difference between the eigenvalue approximation and 
the  eigenvalue computed with Matlab's eigs command. The right panel displays the respective loss of linear dependency of the vectors $w_1$ and $w_2$.
}
\label{conver}
\end{figure}
\begin{figure}
\centering
\includegraphics[width=99.5mm, height=39.5mm]{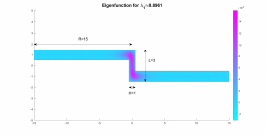}
\caption{Plot of the eigenvector inside the Z-formed  waveguide corresponding to the eigenvalue
$\lambda_1 \approx 8.8961$.}
\label{Fig: Image}
\end{figure}

\section{Conclusions} 
With an approximate eigenvector available for a generalized eigenvalue problem,
all conceivable eigenvalue estimates were given in terms of the quotient function. These estimates contain
an eigenvalue in the normal case while in the non-normal case there is no quarantee of that the estimates
are of use. A change of inner-product may be required. In the self-adjoint case rules were given for choosing
high quality
estimates. Using these estimates in eigenvector approximation yield seemingly the fastest possible quotient iterations
 to compute eigenpairs. Quotient iterations require very good starting vectors for a reliable and predictable convergence
behaviour. To this end variational descent method were devised and combined with preconditioning strategies.

\section*{Appendix A: Derivation of Rayleigh and optimal quotients}\label{secA1}
Using the inner-product $(\cdot,\cdot)_P$, assume having an approximate unit eigenvector  vector $x\in \setC^n$ 
for the linear eigenvalue problem
\eqref{gen}. To produce an eigenvalue estimate $\lambda$ 
in terms of a quotient, impose 
$$(Mx,z)_P=\lambda (Nx,z)_P,$$
  where the applied unit vector $z\in \setC^n$ should be 
  chosen with care. Since left and right eigenvectors need not  be related, the choice $z=x$ is not arguable in general. Taking  $z=\frac{Nx}{\|Nx\|_P}$
  gives the Rayleigh quotient \eqref{ray}.
  Imposing
$$\max_{z \in \setC^n,\, \|z\|_P=1} \{|(z,w_1)_P|^2+|(z,w_2)_P|^2\}$$
with $w_1=\frac{Mx}{\|Mx\|_P}$
and $w_2=\frac{Nx}{\|Nx\|_P}$ gives
the optimal quotient \eqref{qtti}. 
For a careful reasoning behind this suggestion, see  \cite[Section 2.1]{HUKO}.

\section*{Appendix B: Unbounded case}\label{App}
Assume $M$ and $N$ are densely defined linear operators in
a complex separable Hilbert space $\mathcal{H}$ equipped with an inner-product $(\cdot,\cdot)$ such that
the intersection $D(M)\cap D(N)$ of their domains 
is dense. If the quadratic form
\begin{equation}\label{cft}
  x \longmapsto (Mx,Nx),\; \mbox{for}\; x \in D(M)\cap D(N),
\end{equation}
is real valued, then the eigenvalue problem
\eqref{gen} is said to be symmetric. If, moreover, the spectrum is real, then
the eigenvalue problem is called self-adjoint.
In the self-adjoint case 
the Cayley transformation reads
\begin{equation}\label{sssseq}
U=(M+iN)(M-iN)^{-1}
\end{equation}
by being a unitary operator on $\mathcal{H}$.  
 
\begin{theorem} Assume \eqref{gen} is self-adjoint.
  Then $N(M-iN)^{-1}$ is a bounded normal operator on $\mathcal{H}$ and 
  \begin{equation}\label{sssseqe}
  \inf_x \frac{\|Mx\|}{\|Nx\|}=
  \sqrt{\frac{1}{\|N(M-iN)^{-1}\|^2}-1}.
\end{equation}
\end{theorem}

\begin{proof}
  We have a unitary $U=(M+iN)(M-iN)^{-1}=(M-iN+2iN)(M-iN)^{-1}$ and therefore
  $$N(M-iN)^{-1}=\frac{1}{2i}(U-I)$$ is a bounded operator. It is clearly
  normal as well. 
  
  For any $x \in D(M)\cap D(N)$ there holds
\begin{equation}\label{sef}
  \|(M\pm iN)x\|^2=\|Mx\|^2+\|Nx\|^2
  \end{equation}
since $(Mx,Nx)\in \setR$
  Divide both sides in \eqref{sef} by
  $\|Nx\|^2$ to have
  $\frac{\|(M+iN)x\|^2}{\|Nx\|^2}=\frac{\|Mx\|^2}{\|Nx\|^2}+1.$
  Thereby
$$1 \leq  \inf_x  \frac{\|(M+iN)x\|^2}{\|Nx\|^2}=\inf_y
  \frac{\|(M+iN)(M-iN)^{-1}y\|^2}{\|N(M-iN)^{-1}y\|^2}=
  \inf_{\|x\|=1}\frac{1}{\|N(M-iN)^{-1}x\|^2} $$
  by using $\|(M+iN)(M-iN)^{-1}y\| =\|y\|$.
\end{proof}

We have the following variational
principle for locating an eigenvalue nearest to a given point $s \in \setR$.

\begin{corollary}\label{vara}
  Assume \eqref{gen} is self-adjoint with $M$ closable and
  $\infty \in \rho (M,N)$.\footnote{The assumption $\infty \in \rho (M,N)$ means that either $N$ is a bounded invertible operator or
  $N$ with the domain $\in D(M)\cap D(M)$ is one-to-one and onto such that
  $N^{-1}$ is bounded.}
If $s\in \setR$,  then
 $$\inf_x \frac{\|(M-sN)x\|}{\|Nx\|}=\min_{\lambda\in \Lambda(M-sN,N)}|\lambda |.$$
  \end{corollary}

\begin{proof}
We have $$N(M-sN-iN)^{-1}= ((M-sN-iN)N^{-1})^{-1}=
   (MN^{-1}-sI-iI)^{-1}.$$
   Now $MN^{-1}-sI-iI$ is invertible. Observe 
that $(Mx,Nx)=(MN^{-1}y,y)$, with $x=N^{-1}y$, is real.
   Since $y\in \mathcal{H}$ has no constraints and $M$ is closable, we
   can conclude $MN^{-1}-sI$ is a bounded self-adjoint operator on $\mathcal{H}$. Consequently, $MN^{-1}-sI-iI$ is normal, so that 
   $\|N(M-sN-iN)^{-1}\|$
  equals the reciprocal of the distance of $i$ to the spectrum of $MN^{-1}-sI$.
  Then use the fact that $\Lambda(MN^{-1},I) =\Lambda(M,N)$.
  Thereby the right hand-side of \eqref{sssseqe}
yields the distance.
\end{proof}

\end{document}